\documentclass[12pt]{article}
\usepackage{amsmath,amsthm,bbm,amssymb}

\usepackage[all]{xy}
\usepackage[active]{srcltx}
\newcommand{\F}{\mathbb F}
\newcommand{\C}{\mathbb C}
\newcommand{\R}{\mathbb R}

\newcommand{\arr}[1]{\xymatrix{{}\ar@{|->}[r]^{#1}&{}}}
\renewcommand{\a}[1]{\xymatrix{{}\ar@{->}[r]^{#1}&{}}}

\renewcommand{\ge}{\geqslant}

\newtheorem{theorem}{Theorem}
\newtheorem{lemma}[theorem]{Lemma}

\begin{document}
\title{Topological classification of sesquilinear forms: reduction to the nonsingular case}
\author{\emph{Carlos M. da Fonseca}\\
Department of Mathematics, Kuwait University\\ Safat 13060, Kuwait,
carlos@sci.kuniv.edu.kw
\and
\emph{Tetiana Rybalkina\qquad Vladimir V.
Sergeichuk}\\{Institute of Mathematics,
Tereshchenkivska 3, Kiev, Ukraine}\\
rybalkina\_t@ukr.net\qquad
sergeich@imath.kiev.ua}
\date{}

 \maketitle
\begin{abstract}
Two sesquilinear forms
$\Phi:\C^m\times\C^m\to \C$ and
$\Psi:\C^n\times\C^n\to \C$
are called topologically equivalent if there exists a homeomorphism $\varphi :\C^m\to \C^n$ (i.e., a continuous bijection whose inverse is also a continuous bijection) such that
$\Phi(x,y)=\Psi(\varphi (x),\varphi (y))$ for all $x,y\in \C^m$.
R.A. Horn and V.V. Sergeichuk in 2006 constructed a regularizing decomposition of a square complex matrix $A$; that is, a direct sum
$SAS^*=R\oplus
J_{n_1}\oplus\dots\oplus J_{n_p}$, in which $S$ and $R$ are nonsingular and each $J_{n_i}$ is
the $n_i$-by-$n_i$ singular Jordan
block. In this paper, we prove that $\Phi$ and $\Psi$ are topologically equivalent if and only if the regularizing decompositions of their matrices coincide up to permutation of the singular summands $J_{n_i}$ and replacement of $R\in\C^{r\times r}$ by a nonsingular matrix $R'\in\C^{r\times r}$ such that $R$ and $R'$ are the matrices of topologically equivalent forms $\C^r\times\C^r\to \C$. Analogous
results  for bilinear forms over $\C$ and over $\R$ are also
obtained.

\emph{Keywords:} Topological equivalence; Regularizing decomposition; Bilinear and sesquilinear forms

\emph{MSC:} 15A21; 15B33; 37J40
\end{abstract}

\section{Introduction}

In $1974$, Gabriel \cite{gab} reduced the problem of classifying
bilinear forms over an arbitrary field $\F$ to the problem of
classifying nonsingular bilinear forms. In this paper, we take an analogous step towards the topological classification of
bilinear and sesquilinear forms, reducing it to the nonsingular case.

Unlike the problem of topological classification of forms, which has not yet been considered, the
problem of topological classification of linear operators has been
thoroughly studied. Kuiper and Robbin \cite{Kuip-Robb, Robb} gave a
criterion for topological similarity of real matrices without
eigenvalues that are roots of $1$. Their result was extended to
complex matrices in \cite{bud1}. The problem of topological
similarity of matrices with an eigenvalue that is a root of $1$ was
also considered by these authors \cite{Kuip-Robb, Robb} as well as by Cappell and Shaneson
\cite{Capp-conexamp, Capp-2th-nas-n<=6,Capp-big-n<6,
Cap+sha,Cap+ste}, and by Hambleton and Pedersen \cite{h-p,h-p1}. The
problem of topological classification was studied for orthogonal
operators \cite{Pardon}, for affine operators
\cite{Blanc,bud,bud1,Ephr}, for M\"obius transformations
\cite{ryb+ser_meb}, for chains of linear mappings \cite{ryb+ser},
for matrix pencils \cite{f-r-s}, for oriented cycles of linear
mappings \cite{ryb+ser1}, and for quiver representations \cite{lop}.

A pair $(U,\Phi)$ consisting of a vector space $U$ and a bilinear
form $\Phi$ is called by Gabriel \cite{gab} a \emph{bilinear space}.
Similarly, we call a pair $(U,\Phi)$ a \emph{sesquilinear space} if
$\Phi$ is a sesquilinear form. A pair $(U,\Phi)$ is \emph{singular} or
\emph{nonsingular} if $\Phi$ is so. Two spaces $(U,\Phi)$ are
$(V,\Psi)$ are \emph{isomorphic} if there exists a linear bijection
$\varphi :U\to V$ such that
\begin{equation}\label{fdr}
\Phi(x,y)=\Psi(\varphi (x),\varphi (y))\, ,\qquad\text{for all
}x,y\in U.
\end{equation}
The \emph{direct sum} of pairs is the pair
\[
(U,\Phi)\oplus(V,\Psi):=(U\oplus V,\Phi\oplus\Psi).
\]
A pair is \emph{indecomposable} if it is not isomorphic to a direct sum of pairs with vector spaces of smaller sizes.

Let vector spaces $U$ and $V$ be also topological spaces. For
example, they are subspaces of $\C^m:=\C\oplus\dots\oplus\C$ ($m$
summands) with a usual topology. We say that $(U,\Phi)$ and
$(V,\Psi)$ are \emph{topologically equivalent} if there exists a
homeomorphism $\varphi :U\to V$, i.e., a continuous bijection whose
inverse is also a continuous bijection, such that \eqref{fdr} holds.

The main result of the paper is
the following theorem, which is
proved in Section \ref{ssw}.

\begin{theorem}\label{jus}
Let $\F$ be $\C$ or $\R$. Let $(\F^m,\Phi)$ and $(\F^n,\Psi)$ be two
bilinear or two sesquilinear spaces that are topologically
equivalent. Suppose that
\begin{align*}
(\F^m,\Phi)&=(U_0,\Phi_0)\oplus(U_1,\Phi_1)
\oplus\dots\oplus(U_r,\Phi_r)\\
(\F^n,\Psi)&=(V_0,\Psi_0)\oplus(V_1,\Psi_1)
\oplus\dots\oplus(V_s,\Psi_s),
\end{align*}
where $(U_0,\Phi_0)$ and $(V_0,\Psi_0)$ are nonsingular and the
other summands are indecomposable and singular. Then $m=n$, $r=s$,
$(U_0,\Phi_0)$ and $(V_0,\Psi_0)$ are topologically equivalent, and,
after a suitable reindexing, each $(U_i,\Phi_i)$ is isomorphic to
$(V_i,\Psi_i)$.
\end{theorem}

The equality $m=n$ in Theorem \ref{jus}
holds due to
the following statement:
\begin{equation}\label{svt}
\text{if $\F\in\{\C,\R\}$ and $\F^m$ is homeomorphic to $\F^n$, then
$m=n$}
\end{equation}
(see \cite[Corollary 19.10]{Bred} or \cite[Section 11]{McCl}).

Let us reformulate Theorem \ref{jus} in a matrix form. Each bilinear
or sesquilinear space $(\F^m,\Phi)$ can be given by the pair
$(\F^m,A)$, in which $A$ is the matrix of $\Phi$ in the standard
basis. Changing the basis, we can reduce $A$ by congruence
transformations $SAS^T$ with nonsingular $S\in\F^{n\times n}$ if
$\Phi$ is bilinear, or by *congruence transformations $SAS^*$ with
nonsingular $S\in\F^{n\times n}$ if $\Phi$ is sesquilinear.

Each square matrix $M$ over $\F$ is congruent (resp., *congruent) to a direct sum
\begin{equation}\label{li5}
R\oplus J_{n_1}\oplus\dots\oplus J_{n_p}\, ,\qquad\text{with a
nonsingular $R$,}
\end{equation}
in which the matrix $R$ is uniquely determined by $M$  up to
congruence (resp., *congruence) and the $n_i$-by-$n_i$ singular
Jordan blocks $J_{n_i}$ are uniquely determined  up to permutation;
see Theorem \ref{t3}(a). Horn and Sergeichuk \cite{h-s_lin} called
the sum \eqref{li5} the \emph{regularizing decomposition} of $M$,
$J_{n_1},\dots, J_{n_p}$ the \emph{singular summands}, and the
matrix $R$ the \emph{regular part} of $M$. They gave an algorithm
for constructing \eqref{li5} by $M$.

We say that two matrices $A,B\in\F^{n\times n}$ are
\emph{topologically congruent} (resp., \emph{topologically
*congruent}) if the bilinear (resp., sesquilinear) spaces $(\F^n,A)$
and $(\F^n,B)$ are topologically equivalent.

The next theorem is the matricial analogue of Theorem \ref{jus}.

\begin{theorem}\label{jye}
Two square matrices over $\F\in\{\C,\R\}$ are topologically congruent $($resp., *congruent$)$  if and only if their regularizing decompositions coincide up to topological congruence $($resp., *congruence$)$ of their regular parts and permutations of direct summands.
\end{theorem}

The regularizing decomposition \eqref{li5} is the first step towards
reducing a matrix to its canonical form under congruence and
*congruence. Canonical forms under congruence and *congruence over
any field $\F$ of characteristic not 2 were given by Sergeichuk
\cite{ser_izv} (see also \cite{h-s_con}) up to classification of
quadratic and Hermitian forms over finite extensions of $\F$. They
were latter simplified for the case of complex matrices by Horn and
Sergeichuk \cite{hor-ser}. An alternative proof that the canonical
matrices from \cite{hor-ser} are indeed canonical was given by Horn
and Sergeichuk \cite{hor-ser_can}. These authors gave the proof only
for nonsingular matrices, which was sufficient due to the uniqueness
of regularizing decomposition (Theorem \ref{t3}(a)).

\section{The regularizing algorithm}

In this section, we recall the regularizing algorithm for matrices
under congruence and *congruence, which was constructed by Horn and
Sergeichuk \cite{h-s_lin}. An analogous regularization algorithm for
matrix pencils was constructed by Van Dooren \cite{doo}.

Let $\mathbb F$ be any field  with a fixed involution $a\mapsto
\tilde{a}$, which can be the identity. We say that a form is
\emph{$^{\star}$\!sesquilinear} (we use a five-pointed star) if it
is sesquilinear with respect to this involution. The transformation
$A\mapsto SAS^{\star}$ of $A\in\F^{n\times n}$, in which
$S\in\F^{n\times n}$ is nonsingular and $S^{\star}:=\tilde S^T$, is
called the \emph{$^{\star}$\!congruence transformation}. We remark
that {$^{\star}$congruence transformations over $\F=\C$ are
*congruence transformations if the involution $a\mapsto \tilde{a}$
is the complex conjugation, and they are congruence transformations
if the involution is the identity.

We denote by $0_n$ the zero matrix of size $n\times n$, $n\ge 0$,
assuming that when $n=0$ we formally have an empty square matrix.

Let $A$ be a singular square matrix over $\mathbb
F$. We reduce it by
$^{\star}$congruence transformations
as follows:
\begin{align}\label{new1a}
A\,&\longmapsto\,SA=\begin{bmatrix}
A_1\\0
\end{bmatrix}\!\! \begin{matrix}
\\ \}{\scriptstyle m_1}
 \end{matrix}
\quad
\begin{matrix} \text{($S$ is nonsingular and
the rows of}\\ \text{$A_1$
are linearly independent)}
\end{matrix}
            \\ \label{new1b}
 &\longmapsto\,\begin{bmatrix} A_1\\0
\end{bmatrix}S^{\star}=\left[\begin{array}{c|c}
B&C\\\hline 0&0_{m_1}
\end{array}\right]\quad
\text{($S$ is the same and
$B$ is square)}
           \\ \label{new1c}
  &\longmapsto\, (S_1\oplus
I_{m_1})\left[\begin{array}{c|c}
B&C\\\hline 0&0_{m_1}
\end{array}\right](S_1^{\star}\oplus
I_{m_1})
 =\left[\begin{array}{cc|c}
D & E & C_1 \\
F&A_2&0\\\hline
\multicolumn{2}{c|}0 &
0_{m_1}
 \end{array}\right]
\!\!
 \begin{matrix}
\}{\scriptstyle m_2} \\ \\ \}{\scriptstyle m_1}
 \end{matrix}
\end{align}
in which $D$ and $A_2$ are square, $S_1$ is nonsingular, and the
rows of $C_1$ are linearly independent. The nonnegative integers
$m_1,m_2$ and the matrix $A_2$ are used in the following theorem.

\begin{theorem}[\cite{h-s_lin}] \label{t3}
Let $\mathbb F$ be a field  with involution, which can be the identity.

{\rm(a)} Each square matrix $A$ over $\mathbb F$ is
$^{\star}$\!congruent to a direct sum
\begin{equation*}\label{1.1}
R\oplus J_{n_1}\oplus\dots\oplus J_{n_p}\, ,\quad\text{with $R$
nonsingular,}
\end{equation*}
in which the matrix $R$ is determined by $A$  uniquely up to $^{\star}$\!congruence, and the
$n_i\times n_i$ singular Jordan
blocks $J_{n_i}$ are
determined uniquely up to permutation.

{\rm(b)} Given a singular
square matrix $A$ over $\mathbb F$.
Apply the reduction
\eqref{new1a}--\eqref{new1c} to
$A$ and get $m_1,m_2,A_2$. Apply this reduction to
$A_2$ and get $m_3,m_4,A_4$, and
so on until obtain a nonsingular
$A_{2t}$:
\begin{equation}\label{new2}
A \Longrightarrow
  \begin{cases}
    \quad A_2 \Longrightarrow
    \\ m_1,m_2
  \end{cases}
  \begin{matrix}
  \!\!\!\!\!\!
  \begin{cases}
    \quad A_4
    \Longrightarrow\cdots \Longrightarrow
    \\ m_3,m_4
  \end{cases}\\
  \phantom{A}
  \end{matrix}  \!\!\!\!\!
 \begin{matrix}
  \begin{cases}
    \text{nonsingular }A_{2t}
    \\ m_{2t-1},m_{2t}.
  \end{cases}\\
  \phantom{A}\\ \phantom{A}
  \end{matrix}
\end{equation}
Then $m_1\ge m_2\cdots\dots\ge m_{2t}$ and $A$ is
$^{\star}$\!congruent to
\begin{equation*}\label{eq9}
A_{2t}\oplus
J_{1}^{[m_{1}-m_{2}]}\oplus
J_{2}^{[m_{2}-m_{3}]} \oplus
\dots\oplus
J_{2t-1}^{[m_{2t-1}-m_{2t}]}\oplus
J_{2t}^{[m_{2t}]},
\end{equation*}
in which
$J_i^{[m]}:=J_i\oplus\dots\oplus
J_i$ $(m$ summands, in particular,
$J_i^{[0]}=0_0)$. Thus, $A_{2t}$ is the regular part of $A$.

{\rm(c)} If\/ $\mathbb F=\mathbb
C$ or $\R$, then the reduction
\eqref{new2} can be realized by
unitary or orthogonal transformations, respectively, which
improves the numerical stability
of the algorithm.
\end{theorem}

\section{Proof of Theorem \ref{jus}}\label{ssw}

Theorem \ref{jus} is formulated for forms on $\C^n$ or $\R^n$, but
it is more convenient to prove it for forms on unitary or Euclidean
spaces since their subspaces are also unitary or Euclidean, respectively. A unitary space is also called a complex inner product space.
We consider unitary and Euclidean spaces as topological spaces.

Let $\F$ be $\C$ or $\R$, and let
\begin{equation*}\label{krb}
\Phi: U\times U\to \F,\qquad
\Phi': U'\times U'\to \F
\end{equation*}
be two bilinear or two sesquilinear forms on unitary spaces if
$\F=\C$, or two bilinear forms on Euclidean spaces if $\F=\R$. We
suppose that these forms are topologically equivalent, i.e., there
exists a homeomorphism $\varphi:U\to U'$ such that
\begin{equation}\label{kgr}
\Phi (x,y)=\Phi'(\varphi (x),\varphi (y))\, ,\qquad\text{for all
$x,y\in U$}.
\end{equation}

Let $A$ and $A'$ be matrices of $\Phi $ and $\Phi '$ in orthonormal
bases. Applying to  $A$ the reduction \eqref{new1a}--\eqref{new1c}
in which the transforming matrices are unitary if $F=\C$ or
orthogonal if $F=\R$, we obtain $m_1$, $m_2$ and $A_2$. Applying it
to $A'$, we obtain $m'_1$, $m'_2$, and $A'_2$. We need to prove that
\begin{equation}\label{yep}
m_1=m_1',\quad m_2=m_2',\quad \text{and  $A_2$ and $A_2'$ are
topologically $^{\star}$congruent}
\end{equation}
(that is, $A_2$ and $A_2'$ are topologically congruent if the forms
$\Phi $ and $\Phi '$ are bilinear; they are topologically *congruent
if the forms are sesquilinear).

Let $S$ be a unitary matrix if $\F=\C$ or an orthogonal matrix if $\F=\R$ such that $SA$ has the form given in \eqref{new1a}. Then
\begin{equation}\label{kie}
SAS^{\star}=\left[\begin{array}{c|c}
B&C\\\hline 0&0_{m_1}
\end{array}\right]\quad
\text{(the rows of $[B\,C]$
are linearly independent)}
\end{equation}
is the matrix of $\Phi$ in a new orthonormal basis.

The basis vectors that correspond to the second horizontal strip of \eqref{kie} generate the vector space
\begin{equation*}\label{jyw}
L:=\{x\in U\,|\,\Phi (x,U)=0\},
\end{equation*}
which is called the \emph{left kernel}, or the left radical, of
$\Phi $. Denote by $L'$ the left kernel of $\Phi '$. If $x\in L$,
then $\Phi' (\varphi (x),U')= \Phi (x,U)=0$ by \eqref{kgr}, hence
$\varphi(L)\subset L'$. The inclusion $\varphi(L)\supset L'$ holds
too since, for each $x'\in L'$ and setting $x:=\varphi ^{-1}(x')$ we
have
$$\Phi (x,U)=\Phi'(\varphi (x),\varphi(U))=\Phi'(x',U')=0\, ,$$ and so
$x\in L$. Thus,
\begin{equation}\label{lwp}
\varphi(L)=L',
\end{equation}
which proves the first equality in \eqref{yep} due to \eqref{svt}.

Let $S_1$ be a unitary matrix if $\F=\C$ or an orthogonal matrix if $\F=\R$ such that
\[
S_1C=\begin{bmatrix}
       C_1\\
       0\\
     \end{bmatrix}
\quad \text{(the rows of $C_1$ are linearly independent)}.
\]
Then
\begin{equation}\label{lke}
(S_1\oplus
I_{m_1})\left[\begin{array}{c|c}
B&C\\\hline 0&0_{m_1}
\end{array}\right](S_1^{\star}\oplus
I_{m_1})
 =\left[\begin{array}{cc|c}
D & E & C_1 \\
F&A_2&0\\\hline
\multicolumn{2}{c|}0 &
0_{m_1}
 \end{array}\right]
\!\!
 \begin{matrix}
\}{\scriptstyle m_2} \\ \\ \}{\scriptstyle m_1}
 \end{matrix}
\end{equation}
is the matrix of $\Phi $ in a new orthonormal basis.

Since the basis vectors that correspond to the columns of $C_1$ generate $L$, the basis vectors that correspond to the second and third horizontal strips of the right hand side matrix in \eqref{lke} generate the vector space
\begin{equation*}\label{der}
K:=\{x\in U\,|\,\Phi(x,L)=0\}.
\end{equation*}
Analogously, define $K':=\{x'\in U'\,|\,\Phi'(x',L')=0\}$.

By \eqref{kgr} and \eqref{lwp}, for each $x\in K$ we have
$$\Phi'(\varphi (x),L')=\Phi'(\varphi (x),\varphi (L))=\Phi(x,L)=0\,
,$$ and consequently $\varphi (K)\subset K'$. The inclusion
$\varphi(K)\supset K'$ holds too since for each $x'\in K'$ and
$x:=\varphi ^{-1}(x')$, we have $$\Phi (x,L)=\Phi'(\varphi
(x),\varphi(L))=\Phi'(x',L')=0\, ,$$ and so $x\in K$. Thus,
\begin{equation*}\label{kyt}
\varphi(K)=K'.
\end{equation*}
By \eqref{svt}, $$m_2=\dim U-\dim K=\dim U'-\dim K'=m_2'\, ,$$ which
proves the second equality in \eqref{yep}.

Since the basis in $U$ is orthonormal, the basis vectors that
correspond to the second horizontal strip of the right hand side
matrix in \eqref{lke} generate the vector space
\[
L^{\bot}_K:=\{x\in K\,|\,(x,L)=0\},
\]
which is the orthogonal complement of $L$ in $K$. Analogously, write
$L^{\prime\bot}_{K'}:=\{x'\in K'\,|\,(x',L')=0\}.$ Then $
K=L^{\bot}_{K}\oplus L$ and $K'=L^{\prime\bot}_{K'}\oplus L'. $

Define the maps that
are the compositions of three maps:
\begin{equation}\label{ka}
\begin{matrix}
\psi:&
L^{\bot}_K\a{\iota}L^{\bot}_{K}\oplus  L\a{\varphi }L^{\prime\bot}_{K'}\oplus L'\a{\pi' }L^{\prime\bot}_{K'}\\
\psi':&
L^{\prime\bot}_{K'}\a{\iota'}L'^{\bot}_{K'}\oplus L'\a{\varphi^{-1} }L^{\bot}_{K}\oplus  L\a{\pi }L^{\bot}_{K}
\end{matrix}
\end{equation}
where $\iota,\iota'$ are the injections and $\pi ,\pi '$ are the orthogonal projections.

\begin{lemma}\label{jje}
The map $\psi: L^{\bot}_{K}\to L^{\prime\bot}_{K'}$ is a homeomorphism and $\psi^{-1}=\psi'$.
\end{lemma}

\begin{proof}
By \eqref{kgr}, for all $x,y\in U$
\[
\Phi'(\varphi (x),\varphi (y))=\Phi (x,y)=\Phi (x+L,y)=\Phi '(\varphi (x+L),\varphi (y)),
\]
hence
$
\Phi'(\varphi (x+L)-\varphi (x),\varphi (y))=0.
$
Since $\varphi $ is a surjection,
each element of $U'$ is represented in the form $\varphi (y)$, and so $\varphi (x+L)-\varphi (x)\subset L'$. Thus,
$\varphi (x+L)\subset \varphi (x)+L'$, which implies
\begin{equation}\label{lke!}
 \varphi (x+L)=\varphi (x)+L',\qquad\text{for all $x\in U$}
\end{equation}
because $\varphi $ is a surjection.

Let $x\in L^{\bot}_{K}$ and $\varphi(x)=x'+l'$, where $x'\in
L^{\prime\bot}_{K'}$ and $l'\in L'$. By \eqref{lke!}, $\varphi
(x+L)=x'+l'+L'=x'+L'$, and so there exists $l\in L$ such that
$\varphi (x+l)=x'$. By \eqref{ka},
\begin{align*}
\psi&:\ x\arr{\iota}x+0\arr{\varphi}x'+l'\arr{\pi'}x'\\ \nonumber
\psi'&:\ x'
\arr{\iota'}x'+0\arr{\varphi^{-1}}x+l\arr{\pi}x.
\end{align*}
Hence $\psi'\psi =1$. Analogously, $\psi \psi' =1$, and so
$\psi^{-1}=\psi'$. Since
$\iota,\iota',\pi,\pi',\varphi,\varphi^{-1}$ are continuous, $\psi$
and $\psi'$ are continuous too, which proves that $\psi$ is a
homeomorphism.
\end{proof}

For all $x,y\in L^{\bot}_{K}$, we write $\varphi (x)=\psi(x)+l'$ and
$\varphi (y)=\psi(y)+l''$, in which $l',l''\in L'$. It follows from
the zeros in the matrix \eqref{lke} that
\begin{equation}\label{dlt}
\begin{split}
  \Phi(x,y) &= \Phi'(\varphi (x),\varphi (y)) \\
   &= \Phi'(\psi(x)+l',\psi (y)+l'') \\
   &= \Phi'(\psi(x),\psi(y))+\Phi'(\psi(x),l'')+\Phi'(l',\psi(y))+
   \Phi'(l',l'') \\
   &= \Phi'(\psi(x),\psi(y)).
\end{split}
\end{equation}

Let
\begin{equation*}\label{dr5}
\Phi_2: L^{\bot}_{K}\times L^{\bot}_{K}\to \F,\qquad
\Phi'_2: L^{\prime\bot}_{K'}\times L^{\prime\bot}_{K'}\to \F
\end{equation*}
be the restrictions of $\Phi$ and $\Phi'$. By Lemma \ref{jje} and
\eqref{dlt}, these forms are topologically equivalent. Moreover,
$A_2$ and $A_2'$ are their matrices in the orthonormal bases. We
have proved \eqref{yep}.

In the same way, we apply to $A_2$ and $A_2'$ the reduction \eqref{new1a}--\eqref{new1c} in which the transforming matrices  are unitary if $\F=\C$ or  orthogonal if $\F=\R$. We obtain the numbers $m_3=m_3',
m_4=m_4'$ and topologically equivalent forms $\Phi_4,\Phi'_4$ with matrices $A_4,A_4'$
in orthonormal bases (see \eqref{new2}). We repeat this reduction until obtain topologically equivalent forms $\Phi_{2t},\Phi'_{2t}$ with nonsingular matrices $A_{2t},A_{2t}'$.

By Theorem \ref{t3}(b), there exist bases of the spaces $U$ and $U'$, in which the matrices of $\Phi$ and $\Phi'$ have the form
\begin{align*}
&A_{2t}\oplus
J_{1}^{[m_{1}-m_{2}]}\oplus
J_{2}^{[m_{2}-m_{3}]} \oplus
\dots\oplus
J_{2t-1}^{[m_{2t-1}-m_{2t}]}\oplus
J_{2t}^{[m_{2t}]}
        \\
&A'_{2t}\oplus
J_{1}^{[m_{1}-m_{2}]}\oplus
J_{2}^{[m_{2}-m_{3}]} \oplus
\dots\oplus
J_{2t-1}^{[m_{2t-1}-m_{2t}]}\oplus
J_{2t}^{[m_{2t}]}
\end{align*}
which completes the proof of Theorem \ref{jus}.

\end{document}